\documentclass[11pt,reqno,tbtags,a4paper]{amsart}
\usepackage{amssymb}
\usepackage{mathabx}  
\usepackage{xpunctuate}
\usepackage{url}
\usepackage[square,numbers]{natbib}
\bibpunct[, ]{[}{]}{;}{n}{,}{,}

\title
{On semi-restricted Rock, Paper, Scissors}

\date{22 February, 2024;
revised 26 March 2024}

\author{Svante Janson}
\thanks{Supported by the Knut and Alice Wallenberg Foundation
and
the Swedish Research Council
}
\address{Department of Mathematics, Uppsala University, PO Box 480,
SE-751~06 Uppsala, Sweden}
\email{svante.janson@math.uu.se}

\subjclass[2020]{} 

\numberwithin{equation}{section}

\renewcommand\le{\leqslant}
\renewcommand\ge{\geqslant}

\allowdisplaybreaks

\theoremstyle{plain}
\newtheorem{theorem}{Theorem}[section]
\newtheorem{lemma}[theorem]{Lemma}

\newtheorem{conj}[theorem]{Conjecture}

\theoremstyle{definition}

\newcommand\xqed[1]{%
    \leavevmode\unskip\penalty9999 \hbox{}\nobreak\hfill
    \quad\hbox{#1}}

\newtheorem{exampleqqq}[theorem]{Example}
\newenvironment{example}{\begin{exampleqqq}}
  {\xqed{$\triangle$}\end{exampleqqq}}

\newtheorem{remarkqqq}[theorem]{Remark}
\newenvironment{remark}{\begin{remarkqqq}}
  {\xqed{$\triangle$}\end{remarkqqq}}

\newtheorem{problem}[theorem]{Problem}

\newtheorem*{ack}{Acknowledgement}

\theoremstyle{remark}

\newcounter{dummy}
\makeatletter
\newcommand\myitem[1][]{\item[#1]\refstepcounter{dummy}\def\@currentlabel{#1}}
\makeatother

\newenvironment{romenumerate}[1][-10pt]{
\addtolength{\leftmargini}{#1}\begin{enumerate}
 \renewcommand{\labelenumi}{\textup{(\roman{enumi})}}%
 }{\end{enumerate}}

\newenvironment{Romenumerate}[1][-10pt]{
\addtolength{\leftmargini}{#1}\begin{enumerate}
 \renewcommand{\labelenumi}{\textup{\Roman{enumi}.}}%
 }{\end{enumerate}}

\newcounter{oldenumi}
{\setcounter{oldenumi}{\value{enumi}}
\begin{romenumerate} \setcounter{enumi}{\value{oldenumi}}}
{\end{romenumerate}}

\newcounter{thmenumerate}

\newcounter{xenumerate}   

\newcommand\pfitemx[1]{\par#1:}
\newcommand\pfitemref[1]{\pfitemx{\ref{#1}}}

\newcommand{\refT}[1]{Theorem~\ref{#1}}
\newcommand{\refTs}[1]{Theorems~\ref{#1}}

\newcommand{\refL}[1]{Lemma~\ref{#1}}

\newcommand{\refS}[1]{Section~\ref{#1}}
\newcommand{\refSs}[1]{Sections~\ref{#1}}
\newcommand{\refSS}[1]{Section~\ref{#1}}
\newcommand{\refP}[1]{Problem~\ref{#1}}

\newcommand{\refE}[1]{Example~\ref{#1}}

\newcommand{\refF}[1]{Figure~\ref{#1}}

\newcommand\set[1]{\ensuremath{\{#1\}}}
\newcommand\bigset[1]{\ensuremath{\bigl\{#1\bigr\}}}

\newcommand\xpar[1]{(#1)}
\newcommand\bigpar[1]{\bigl(#1\bigr)}
\newcommand\Bigpar[1]{\Bigl(#1\Bigr)}

\newcommand\lrpar[1]{\left(#1\right)}
\newcommand\bigsqpar[1]{\bigl[#1\bigr]}
\newcommand\sqpar[1]{[#1]}
\newcommand\Bigsqpar[1]{\Bigl[#1\Bigr]}

\newcommand\cpar[1]{\{#1\}}
\newcommand\bigcpar[1]{\bigl\{#1\bigr\}}

\newcommand\bigabs[1]{\bigl\lvert#1\bigr\rvert}

\def\rompar(#1){\textup(#1\textup)}    

\def\xexp(#1){e^{#1}}
\newcommand\ceil[1]{\lceil#1\rceil}

\newcommand\ntoo{\ensuremath{{n\to\infty}}}

\newcommand\bmin{\land}

\newcommand\punkt{\xperiod}    
\newcommand\iid{i.i.d\punkt}    
\newcommand\ie{i.e\punkt}
\newcommand\eg{e.g\punkt}

\newcommand\whp{w.h.p\punkt}

\newcommand\ii{\mathrm{i}}

\newcommand{\tend}{\longrightarrow}
\newcommand\dto{\overset{\mathrm{d}}{\tend}}

\newcommand\Op{O_{\mathrm p}}
\newcommand\Olp{O_{L^p}}
\newcommand\Olpnnn{\Olp\bigpar{n\qqq}}

\newcommand\bbR{\mathbb R}
\newcommand\bbC{\mathbb C}

\newcommand\E{\operatorname{\mathbb E}{}} 
\renewcommand\P{\operatorname{\mathbb P{}}}

\newcommand\Var{\operatorname{Var}}
\newcommand\Cov{\operatorname{Cov}}

\newcommand\ga{\alpha}
\newcommand\gb{\beta}

\newcommand\gD{\Delta}

\newcommand\gS{\Sigma}
\newcommand\gss{\sigma^2}

\newcommand\gTH{\Theta}

\newcommand\eps{\varepsilon}
\renewcommand\phi{\xxx}  

\newcommand\cS{{\mathcal S}}

\newcommand\tW{\widetilde W}

\newcommand\matrixx[1]{\begin{pmatrix}#1\end{pmatrix}}

\newcommand\qw{^{-1}}

\newcommand\qq{^{1/2}}
\newcommand\qqw{^{-1/2}}
\newcommand\qqq{^{1/3}}

\newcommand\intoo{\int_0^\infty}

\newcommand\dd{\,\mathrm{d}}

\newcommand\rhs{right-hand side}

\newcommand\RRR{\textsf R}
\newcommand\NNN{\textsf N}
\newcommand\NN{\mathbf N}
\newcommand\qNN{\NN'}
\newcommand\qN{N'}
\newcommand\qqNN{\NN''}
\newcommand\qqN{N''}
\newcommand\qM{M'}
\newcommand\qqM{M''}
\newcommand\Xm{X_{\max}}
\newcommand\Zm{Z_{\max}}
\newcommand\be{\mathbf{e}}
\newcommand\bff{\mathbf{f}}
\newcommand\bxi{\boldsymbol\xi}
\newcommand\bfeta{\boldsymbol\eta}
\newcommand\bzeta{\boldsymbol\zeta}
\newcommand\settt{\set{1,2,3}}

\newcommand\RX{\emph{rock}}
\newcommand\PX{\emph{paper}}
\newcommand\SX{\emph{scissors}}
\newcommand\tr{^{\mathsf{tr}}}
\newcommand\hbf{\widehat{\mathbf{f}}}
\newcommand\hbzeta{\widehat{\bzeta}}
\newcommand\Sop{S^{\mathrm{op}}}
\newcommand\Sgr{S^{\mathrm{gr}}}
\newcommand\SSS{\cS^{\mathrm{gr}}}
\newcommand\SSSop{\cS^{\mathrm{op}}}

\begin{document}

\begin{abstract} 
Spiro, Surya and Zeng (Electron.\ J. Combin.\ 2023) recently studied
a semi-restricted variant of the well-known game Rock, Paper, Scissors;
in this variant the game is played for $3n$ rounds, but one of the two players 
is restricted and has to use each of the three moves exactly $n$ times.
They show that the optimal strategy for the restricted player is the greedy
strategy, and show that it results 
in an expected score for the unrestricted player $\Theta(\sqrt{n})$;
they conjecture, based on numerical evidence, that 
the expectation is $\approx 1.46\sqrt{n}$.
We analyse the result of the strategy further and show that
the average is $\sim c \sqrt{n}$ with 
$c=3\sqrt{3}/2\sqrt{\pi}\doteq1.466$, verifying the conjecture.

The proof is based on considering the case when both players play greedily,
which leads to the same expectation as optimal play; for this case we also
find the asymptotic distribution of the score, and compute its 
variance.
\end{abstract}

\maketitle

\section{Introduction}\label{S:intro}

A semi-restricted variant of the well-known game Rock, Paper, Scissors
(RPS) was recently studied by \citet{RPS}.
In the standard version of RPS, two players simultaneously select one of the
three choices \RX, \PX,  \SX, where \PX{} beats \RX, \SX{} beats \PX, and 
\RX{} beats \SX; if both select the same, the result is a draw.
The game is symmetric, so there is obviously no advantage to any of the
players. It is easy to see that the optimal strategy for both players is to
choose randomly, with equal probability for each choice (see further
\refSS{SSgreedy}).

In the semi-restricted variant in \cite{RPS}, two players \RRR{} (restricted)
and \NNN{} (normal) agree to play $3n$ rounds of RPS for some integer $n$, 
but \RRR{} is
restricted to choose \RX, \PX, and \SX{} exactly $n$ times each, while
\NNN{} plays without restriction.
Clearly, the restriction is a disadvantage for \RRR. (In particular, \NNN{} 
will always win the last round, since \RRR{} then has only one choice, and
\NNN{} knows which one.) How large is this disadvantage?
More precisely, let $S_n$ be the final score of \NNN, defined as the number
of rounds won by \NNN{} minus the number lost.
We assume (as \cite{RPS}) that the objective of both players is the
expectation $\E S_n$, which \NNN{} wants as high as possible, while \RRR{} wants
the opposite.
Semi-restricted RPS is a two-player zero-sum game, and thus 
by the theory of \citet{vN}, each player has an optimal randomized strategy,
see further \eg{} \cite[Chapter 2]{Peres}.
We use $\Sop_n$ to denote the final score when both players use their optimal
strategies. (This is a random variable, since the strategies are randomized.)

The main result of \cite{RPS} is that the unique optimal strategy for 
\RRR{} is to play greedily, \ie, as if each round were the last;
see further \refSS{SSgreedy}. 
(This is far from obvious, and rather surprising.)
It is also shown in \cite{RPS} that with optimal strategies,
the expected gain 
$\E \Sop_n = \Theta(\sqrt n)$, 
and it is asked 
\cite[Question 21]{RPS}
whether $\E \Sop_n\sim c \sqrt n$ for some constant $c>0$ as \ntoo;
\cite{RPS} says further that
numerical calculations for $n\le100$ suggest that this might hold with
$c\approx1.46$.

The main purpose of the present note is to verify this conjecture, 
and to identify the constant.

\begin{theorem}\label{T1}
  For semi-restricted RPS played over $3n$ rounds,
the expected score for \NNN{} with optimal plays for both players
is, as \ntoo,
\begin{align}\label{t1}
  \E \Sop_n \sim
\sqrt{\frac{27n}{4\pi}}
= 
\frac{3\sqrt3}{2\sqrt\pi}\sqrt n
.\end{align}
\end{theorem}

The constant $3\sqrt3/(2\sqrt\pi)\doteq 1.4658$,
which  verifies also the numerical
conjecture in \cite{RPS}.

The optimal strategy for \RRR{} is thus the greedy strategy.
Given that \RRR{} uses this strategy, there are many strategies for \NNN{}
that give the optimal expectation $\E \Sop_n$. One of them is the greedy
strategy for \NNN, but as pointed out to me by 
Sam Spiro [personal communication], the greedy strategy is \emph{not} the
optimal strategy for \NNN; 
see \refSS{SSNNN}.
We let $\Sgr_n$ denote $S_n$ when \emph{both} players play with their greedy
strategies. This is also a random variable, and as just said, we have
\begin{align}\label{=}
  \E\Sgr_n = \E\Sop_n.
\end{align}

The random variable $\Sgr_n$ can be analysed asymptotically using standard
tools from probability theory. 
This is done in \refSs{Spf} and \ref{Spf2} and yields the asymptotics of
$\E\Sgr_n$; \refT{T1} then follows by \eqref{=}.
Moreover, our analysis also yields the asymptotic distribution of $\Sgr_n$,
see \refT{T2}. 

In \refS{SNNN+} we give some partial results on the asymptotic
distribution of $S_n$ if \RRR{} uses the optimal (greedy) strategy
and \NNN{} uses a rather arbitrary strategy, including the case 
$\Sop_n$ when both play optimally. 
We  leave
as an open problem whether $\Sop_n$ and $\Sgr_n$
have the same asymptotic distribution.

In \refS{Swin}, we  discuss the probability that the disadvantaged player
\RRR{} nevertheless wins the game; we compute it for the case that both
players play greedily, 
but leave the case of optimal play for the objective of 
maximizing the probability of winning
as an open problem.

\begin{ack}
  I am grateful to Sam Spiro for pointing out a serious error in a previous
  version, and for showing me \refE{ESpiro}.
I also thank an anonymous referee for helpful comments.
\end{ack}

\section{Preliminaries}\label{Sprel}
\subsection{Notation}\label{SSnot}
The three choices \RX, \PX, \SX{} will be numbered $1,2,3$; thus $i+1$
beats $i\pmod3$.

The random variable
$S(t)$ is the score of \NNN{} after round $t=1,\dots,3n$, i.e., the number
of rounds 
won by \NNN{} so far minus the number of rounds won by \RRR.
As in the introduction,
$S_n:=S(3n)$ is the score at the end of the game.
(Except for $S_n$, we do not show $n$ explicitly in the notation, although
$S(t)$ and many variables introduced below depend on $n$.)

If $X_n$ is a sequence of random variables, and $a_n$ a sequence of (positive)
numbers, we write  $X_n=\Op(a_n)$ if the family $\set{X_n/a_n}$ is 
{bounded in probability} (also called \emph{tight}), \ie, if for every
$\eps>0$ there exists $C$ such that $\P(|X_n|>C a_n)<\eps$ for all $n$.
Furthermore, we write $X_n=\Olp(a_n)$ (where $p>0$ is a parameter) 
if the family $\set{X_n/a_n}$ is bounded in $L^p$, \ie,
$\sup_n\E|X_n/a_n|^p<\infty$.

$N(0,\gss)$ denotes the normal distribution with mean 0 and variance
$\gss\ge0$. More generally, if $\gS$ is a symmetric positive semidefinite
$d\times d$ matrix, then $N(0,\gS)$ is the normal distribution 
with mean 0 and covariance matrix $\gS$; this is a distribution 
of a random vector in $\bbR^d$.

The basis vectors in $\bbR^3$ are denoted 
$\be_1:=(1,0,0)$, $\be_2:=(0,1,0)$, $\be_3:=(0,0,1)$.

We use
$C_p, C_p', C_p''$ for some constants that depend on the parameter $p$.

Unspecified limits are as \ntoo.

\subsection{The greedy strategy}\label{SSgreedy}
Recall that in any two-person zero-sum game, each player has an optimal
strategy which in general is randomized; the different alternatives are
selected with some probabilities chosen such that they maximize
the minimum over all strategies of the opponent of the expected gain; see
\cite{vN} and \eg{} \cite{Peres}.

As said above, it was shown by \citet{RPS} that in semi-restricted RPS, 
the best strategy of \RRR{} is to play greedily, \ie, to analyse each round
separately and use the optimal strategy for the expected score in that
round.
(This is far from obvious, since the best play in one specific round may be
punished by lower expected score in later rounds; nevertheless, \cite{RPS}
shows that the expected later gains by any alternative strategy are offset
by the immediate expected loss.)
This optimal strategy for a single round
is easy to find (as was done in \cite{RPS}):

\begin{romenumerate}
  
\item \label{RRR1}
If \RRR{} still has all three choices available, then
the optimal strategy is (obviously, by symmetry), to choose one of them
randomly, with probability $1/3$ each. 
And the best strategy for \NNN{} is the same. 
(This game was one of the examples in the original paper by \citet{vN}.)
The outcome for \NNN{} is
$-1$, $0$, or $+1$ with probability $1/3$ each.

\item \label{RRR2}
If \RRR{} has only two choices available, say 1 (\RX) and 2 (\PX), then 
the game is described by the matrix in \refF{fig:1}.
\NNN{} should never play 1 (which in this case can lose but never win).
A simple calculation shows \cite{RPS}
that the best strategy for \RRR{} is to play 
1 with probability $1/3$ and 
2 with probability $2/3$; similarly
\NNN{} plays
2 with probability $2/3$ and 
3 with probability $1/3$.
The expected gain for \NNN{} is $1/3$.
\begin{figure}[ht]
  \centering
\begin{tabular}{rrrr}
   & \RX & \PX & \SX
\\
\RX & 0 & 1 & $-1$\\
\PX & $-1$ & 0 & 1
\end{tabular}
  \caption{Score matrix for \NNN{} when \RRR{} is restricted to
    \set{\RX,\PX}; rows show the move by \RRR; columns the move by \NNN.}
  \label{fig:1}
\end{figure}

\item \label{RRR3}
If \RRR{} has only one choice, then \RRR{} has to play that, 
and \NNN{} obviously 
plays the next choice (mod 3) and is sure to win. Gain for \NNN{} is 1.
\end{romenumerate}

\subsection{Strategies for \NNN}\label{SSNNN}
Suppose that \RRR{} plays optimally, \ie, greedily.
Then \RRR{} plays each time with a random move
that depends only on the available moves, and thus on the history of the
moves made by \RRR. However, these moves are not affected by the moves made
by \NNN. Hence, the moves made by \RRR{} will be the same 
regardless of the  strategy chosen by \NNN. It thus follows from the discussion
above of the greedy strategy that the expected gain for \NNN{} will 
be the same for any strategy of \NNN{} that does not do anything stupid
(here and in the sequel meaning making a move that cannot win); 
for example, as long as \RRR{} is
able to make all three moves, the expected gain of each round is 0 for any
strategy of \NNN.
In particular, the expected gain 
for \NNN{} when both players use their
optimal strategies is the same as when both play greedily,
which shows \eqref{=}.
Nevertheless, the greedy strategy is not optimal for \NNN, 
since it may be worse if \RRR{} chooses a different strategy
as shown by the
following simple example.

\begin{example}\label{ESpiro}
(Sam Spiro, personal communication.)
  Suppose that \NNN{} plays with the greedy strategy described above.
If \RRR{} chooses to play (deterministically) $1,2,3,1,\allowbreak2,3,\dots$
for all 
$3n$ rounds, then for all but the last two rounds, the greedy strategy by
\NNN{} makes him play randomly, with probability $1/3$ for each choice, and
therefore the expected gain is 0 for each round. 
Hence the total gain $\E S_n$ will in this case be only $1/3+1=4/3$ (from
the last two rounds), while we know 
from von Neumann's theorem \cite{vN} that \NNN{} has some strategy
guaranteeing an expected gain of at least $\E\Sop_n$ against every strategy
of \RRR.
(Note that $\E\Sop_n>4/3$ at least for large $n$ by \refT{T1}. In fact, it
is can easily be seen from \eqref{ESn} below 
that the inequality
holds for every $n\ge2$.)
\end{example}

It seems likely that the optimal strategy of \NNN{} is very complicated.
See further \refSS{SSNNNop}.

\section{Analysis for the greedy strategies}\label{Spf}
In this section we assume that \RRR{} uses the greedy strategy, which is
known to be optimal.
For simplicity, we assume here that also \NNN{}  uses the greedy strategy.
In fact, most of the analysis is valid for almost any strategy by \NNN; we
discuss the few but important differences in \refS{SNNN+}.

Let $N_{t,i}$ be the number of times that \RRR{} plays $i$ during rounds
$1,\dots,t$. The vector $\NN_t=(N_{t,i})_{i=1}^3$ then evolves as 
a random walk which changes character each time some $N_{t,i}$ hits $n$ and
\RRR{} thus cannot choose $i$ in the future.
We let $T_j$, $j=1,2,3$, be the first time that \RRR{} has used up $j$ of
the three choices;
in particular, $T_3:=3n$, when the game ends.

Since \RRR{} uses the greedy strategy described above, 
$\NN_t$ evolves as follows,
for $t=0,\dots,n$,
starting at $\NN_0=(0,0,0)$:
\begin{Romenumerate}
  \item \label{A1}
A random walk $\NN_0,\dots,\NN_{T_1}$ with increments that are independent
and uniformly chosen from $\set{\be_1,\be_2, \be_3}$,
until
\begin{align}
   T_1&:=\inf\bigset{t: N_{t,i}=n \text{ for some $i\in\set{1,2,3}$}}.
\end{align}
\item \label{A2}
A random walk $\NN_{T_1},\dots,\NN_{T_2}$ with increments chosen independently
and randomly from the remaining two choices by the strategy above; 
for example, if $N_{t,1}$ hits $n$ first, so $N_{T_1,1}=n>N_{T_1,2},N_{T_1,3}$,
then the increments are chosen as $\be_2$ and $\be_3$ with probabilities
$1/3$ and $2/3$. This goes on until
\begin{align}
 T_2&:=\inf\bigset{t: N_{t,i}=n \text{ for at least two  $i\in\set{1,2,3}$}}.
\end{align}

  \item \label{A3}
A deterministic walk $\NN_{T_2},\dots,\NN_{T_3}$ where all increments are $\be_i$
for the only $i$ that still has $N_{t,i}<n$.
  \end{Romenumerate}

The expected gain for \NNN{} is 
$0$ for each step in phase \ref{A1},
$1/3$ for each step in phase
\ref{A2}, and 1 for each step in phase \ref{A3},
so the expected score for \NNN{} is 
\begin{align}\label{ESn}
\E S_n = \E S(3n)=
\E\sqpar{(T_2-T_1)/3+T_3-T_2}.  
\end{align}
We will analyse this more carefully below and also both bound 
and asymptotically describe the random fluctuations. 
We do this by analysing the
constrained random walk $\NN_t$ and
the stopping times $T_1$ and $T_2$ in some detail.
A central role in the analysis is played by the (somewhat arbitrary)
non-random time
\begin{align}\label{t0}
  T_0:= 3n - 3\ceil{n^{2/3}}.
\end{align}

\subsection{Phase \ref{A1}:  until $T_1$}
Let $(\bxi_t)_{t=1}^\infty$ be an \iid{} sequence of random vectors with the
distribution $\P(\bxi_t=\be_i)=1/3$ for $i=1,2,3$.
We may assume that $\NN_t-\NN_{t-1}=\bxi_t$ for $1\le t\le T_1$.
Let 
\begin{align}\label{qN}
\qNN_t=(\qN_{t,i})_{i=1}^3:=\sum_{u=1}^t\bxi_u  ,
\qquad t\ge0; 
\end{align}
thus $\qNN_t=\NN_t$ for $t\le T_1$.
(We may interpret $\bxi_t$ and $\qNN_t$ as how \RRR{} would have played if the
restriction had not existed.)
In particular,  for $t\le T_1$ we have $\qN_{t,i}=N_{t,i}\le n$ for all $i$,
and for $t\ge T_1$ we have
$\max_i\qN_{t,i}\ge\max_i N_{T_1,i}=n$; thus $T_1$ is also the time that
$\max_i\qN_{t,i}$ hits $n$.

At time $T_0$, the central limit theorem shows that 
\begin{align}\label{od1}
\qN_{T_0,i}=\tfrac{1}{3}T_0+\Op(n\qq)  =n -n^{2/3} + \Op\bigpar{n\qq}.
\end{align}
This is less that $n$ for each $i$
\whp{} 
(with high probability, \ie, with probability $1-o(1)$ as \ntoo), 
and thus \whp{} $T_1>T_0$.
More precisely, the Chernoff inequality 
(\eg{} in the version in \cite[Remark 2.5]{SJII}) 
yields
\begin{align}\label{od2}
  \P(T_1\le T_0) \le \sum_{i=1}^3 \P(\qN_{T_0,i}\ge n)
=3\P\bigpar{\qN_{T_0,1}-\tfrac13 T_0 \ge \ceil{n^{2/3}}}
\le  e^{- 2 n^{4/3}/T_0}
\le e^{-n^{1/3}}.
\end{align}
Hence, this probability decreases faster than any polynomial, which means
that we can ignore the event $T_1\le T_0$ also when calculating moments below
(since the random variables we consider all are deterministically $O(n)$).

Similarly, concentrating on the time after $T_0$, define
\begin{align}\label{tor1}
\qM:=\max_{i=1,2,3} \max_{T_0\le t\le 3n} 
\bigabs{\qN_{t,i}-\qN_{T_0,i}-\tfrac13(t-T_0)}
.\end{align}
By classical results on moment convergence in the central limit theorem
together with Doob's inequality 
(since $\qN_{t,i}-\qN_{T_0,i}-\tfrac13(t-T_0)$ is a martingale),
see for example \cite[Theorem 7.5.1, Corollary 3.8.2, and Theorem 10.9.4]{Gut},
we have, for any $p>1$
\begin{align}\label{tor2}
  \E (\qM)^p& 
\le \sum_{i=1}^3 \E\max_{T_0\le t\le  3n}
\bigabs{\qN_{t,i}-\qN_{T_0,i}-\tfrac13(t-T_0)}^p
\notag\\&
\le C_p\sum_{i=1}^3 \E\bigabs{\qN_{3n,i}-\qN_{T_0,i}-\tfrac13(3n-T_0)}^p
\notag\\&
\le C'_p (3n-T_0)^{p/2}
\le C''_p n^{p/3}
.\end{align}
Consequently,
\begin{align}\label{tor3}
  \qM=\Olpnnn
\end{align}
for every $p<\infty$. 
(The case $p\le1$ follows from the case $p>1$ by Lyapounov's inequality
\cite[Theorem 3.2.5]{Gut}.)

We introduce some further notation.
Let, for $i=1,2,3$,
\begin{align}\label{oden4}
  X_i:=\qN_{T_0,i}-\E \qN_{T_0,i}
=\qN_{T_0,i}-\tfrac13T_0.
\end{align}
(If we ignore the minor technical difference between $N_{t,i}$ and
$N'_{t,i}$, these measure thus the deviation from the expectation at time
$T_0$ of the choices made by \RRR.)
Note for later use that
\begin{align}\label{oden0}
  X_1+X_2+X_3=\sum_{i=1}^3 \qN_{T_0,i}-T_0=0.
\end{align}
Furthermore, let
\begin{align}
  \label{oden3}
\Xm:=\max_{i=1,2,3}X_i.
\end{align}
(As we will see in detail below, this largest deviation will give us
a good estimate of the time $T_1$ when \RRR{} runs out of one choice.)

Condition on the event $T_1>T_0$, which has probability $1-o(1)$.
Then, $\qN_{T_0}=N_{T_0}$. Moreover,
we may take  $t=T_1$ in \eqref{tor1} and obtain,  using \eqref{oden4},
\begin{align}\label{tor4}
  N_{T_1,i}=\qN_{T_1,i}
=\qN_{T_0,i}+\tfrac13(T_1-T_0)+O(\qM)
=X_i+\tfrac13T_1+O(\qM).
\end{align}
Hence,
recalling the definitions of $T_1$ and $\Xm$,
\begin{align}\label{tor5}
  n=\max_i N_{T_1,i}= \max_i X_i+\tfrac13T_1+O(\qM)
= \Xm+\tfrac13T_1+O(\qM)
.\end{align}
Consequently,
\begin{align}\label{tor8M}
  T_1  = 3n  - 3\Xm + O(\qM)
\end{align}
and thus, using \eqref{tor3},
\begin{align}\label{tor8}
  T_1 
 = 3n  - 3\Xm + \Olpnnn
.\end{align}
This was derived conditioned on $T_1>T_0$, but by \eqref{od2} and the
comment after it, \eqref{tor8} holds also unconditionally.

Furthermore, for every $i\in\settt$,
by \eqref{tor4} and  \eqref{tor3}, 
\begin{align}\label{tor9}
  N_{T_1,i}-\tfrac13T_1 
=X_i+\Olpnnn
\end{align}
and thus by \eqref{tor8}
\begin{align}\label{tor10}
n-  N_{T_1,i}
=n- \tfrac13T_1 -X_i+\Olpnnn
= \Xm -X_i+\Olpnnn.
\end{align}

Thus, at time $T_1$, when \RRR{} runs out of one of the three choices,
she has approximatively $\Xm-X_i$ left of each other choice $i$.

To find the score in Phase \ref{A1}, 
consider first the score at $T_0$, and condition again on
$T_1>T_0$. Then in each round up to $T_0$, \RRR{} plays normally and thus \RRR{}
and \NNN{} win with probability $1/3$ each, and draw otherwise;
thus $\gD S(t):=S(t)-S(t-1)\in\set{\pm1,0}$ with probability $1/3$ each.
Consequently, the central limit theorem shows that,
since $\E\gD S(t)=0$ and $\Var\gD S(t)=2/3$, and $T_0\sim 3n$,
\begin{align}\label{njord0}
  \frac{S(T_0)}{n\qq} \dto N(0,2),
\qquad \text{as } \ntoo,
\end{align}
together with all moments.
Moreover, since also \NNN{} is assumed to use the optimal strategy, which
for these $t$ means uniformly randomly,
the score in each round 
is independent of the choices made by \RRR, and thus of the
vectors $\NN_t$.
Consequently, $S(T_0)$ is independent of $(X_1,X_2,X_3)$.
We conditioned here on $T_1> T_0$, but in the unlikely event $T_1\le T_0$, 
we may modify $S(T_0)$ 
(similarly as we defined $\qNN$ above)
and define a sum $S'(T_0)$ that is independent of $(X_1,X_2,X_3)$
and satisfies $S'(T_0)=S(T_0)$ whenever $T_1>T_0$, and thus, by \eqref{od2},
(rather coarsely)
\begin{align}\label{skade1}
  S(T_0)=S'(T_0)+\Olpnnn.
\end{align}

For $T_0<t\le T_1$, we still have the same distribution of $\gD S(t)$,
and by the same argument as in \eqref{tor2}, 
if we condition on $T_1>T_0$, then
\begin{align}\label{njord1}
  S(T_1)-S(T_0) = \Olpnnn.
\end{align}
By \eqref{od2}, this holds also unconditionally.

\subsection{Phase \ref{A2}: $T_1$ to $T_2$}

Since the entire game is symmetric under cyclic permutations of the three
choices \RX, \PX, \SX, 
we may for the next phase assume that \RRR{} first uses up all $n$ \RX,
i.e., that $N_{T_1,1}=0$.
Note, however, that the game is not symmetric under odd permutations, so 
having made this assumption, choices 2 (\PX) and 3 (\SX) play
different roles, since 3 beats 2.

By the discussion of the greedy strategy in \refSS{SSgreedy}, 
for $t\in[T_1,T_2)$, \RRR{} should play randomly and choose 2 or 3 with
probabilities 1/3 and 2/3. 
We argue as in the preceding subsection (and therefore omit some details); 
we now let $(\bfeta_t)_1^\infty$ be
an \iid{} sequence of random vectors with $\P(\bfeta_t=\be_i)=p_i$
for $i=1,2,3$, with $(p_1,p_2,p_3)=(0,\frac{1}3,\frac{2}3)$,
and we assume as we may that
$\NN_t-\NN_{t-1}=\bfeta_t$ for $T_1<t\le T_2$.
Let
\begin{align}\label{frej2}
\qqNN_t=(\qqN_{t,i})_{i=1}^3:=\NN_{T_1}+\sum_{u=T_1+1}^t\bfeta_u  ,
\qquad t\ge T_1.
\end{align}
Then $\qqNN_t=\NN_t$ for $T_1\le t\le T_2$.
Let
\begin{align}\label{frej4}
\qqM:=\max_{i=1,2,3} \max_{T_1\le t\le 3n} 
\bigabs{\qqN_{t,i}-\qqN_{T_1,i}-p_i(t-T_1)}
.\end{align}
If we again condition on $T_1> T_0$, we obtain, by conditioning on $T_1$ and
arguing as in \eqref{tor2} and using $3n-T_1<3n-T_0=O(n^{2/3})$,
\begin{align}\label{frej3}
  \qqM=\Olpnnn.
\end{align}
 By \eqref{od2} again,
this holds also unconditionally. 
We obtain from \eqref{frej4} and \eqref{frej3}, taking $t=T_2$,
for every $i$,
\begin{align}\label{frej45}
N_{T_2,i}= \qqN_{T_2,i}=N_{T_1,i}+p_i(T_2-T_1)+\Olpnnn
\end{align}
and thus, by \eqref{tor10},
\begin{align}\label{frej5}
n-N_{T_2,i}&
=  n-N_{T_1,i}-p_i(T_2-T_1)+\Olpnnn
\notag\\&
=  \Xm-X_i-p_i(T_2-T_1)+\Olpnnn.
\end{align}
We have assumed $N_{T_1,1}=n$, and then $T_2$ is the first $t$ such that
$N_{t,2}=n$ or $N_{t,3}=n$.
In particular, \eqref{frej5} implies
\begin{align}\label{frej6}
0 = \min_{i=2,3}\bigpar{n-N_{T_2,i}}
= \min_{i=2,3}\bigpar{ \Xm-X_i-p_i(T_2-T_1)}+\Olpnnn.
\end{align}
Consequently,
\begin{align}\label{frej6b}
  \min_{i=2,3}\bigpar{ \Xm-X_i-p_i(T_2-T_1)}=\Olpnnn.
\end{align}
It follows that also
\begin{align}\label{frej7a}
\min_{i=2,3}p_i\qw\bigpar{ \Xm-X_i-p_i(T_2-T_1)}=\Olpnnn,
\end{align}
which can be written
\begin{align}\label{frej7b}
\min_{i=2,3}p_i\qw\bigpar{ \Xm-X_i}
-(T_2-T_1)
=\Olpnnn.
\end{align}
 Thus
\begin{align}\label{frej8}
T_2-T_1
= \min_{i=2,3}\frac{\Xm-X_i}{p_i}+\Olpnnn.
\end{align}
We repeat that this holds assuming that choice 1 is the first to be used up by
\RRR. 

In this phase, the gain $\gD S(t)$ of \NNN{} has expectation $1/3$ in each round
(and its absolute value is bounded by 1, so all moments are bounded);
moreover, the gains in different rounds are \iid.
Hence, 
similarly to \eqref{tor2} again,
the central limit theorem with moment convergence together with Doob's
inequality yields
\begin{align}\label{njord2}
  S(T_2)-S(T_1)= \tfrac13(T_2-T_1)+\Olpnnn.
\end{align}

\subsection{Phase \ref{A3}: $T_2$ to $T_3$}
This phase is deterministic, and not very fun to play (at least not for
\RRR): \RRR{} has only one choice, and \NNN{} wins every round.
The total gain for \NNN{} in this phase are thus, using \eqref{tor8} and
recalling that $T_3=3n$,
\begin{align}\label{njord3}
  S(T_3)-S(T_2)&
=T_3-T_2=T_3-T_1-(T_2-T_1)
\notag\\&
=3 \Xm-(T_2-T_1)
+\Olpnnn.
\end{align}

\subsection{Collecting the gains}
By \eqref{njord1}, \eqref{njord2}, and \eqref{njord3},
the final score of \NNN{} is
\begin{align}\label{njord4}
S_n=  S(T_3)= S(T_0)+3\Xm-\tfrac{2}{3}(T_2-T_1)+\Olpnnn,
\end{align}
where furthermore
$T_2-T_1$ is given by \eqref{frej8} when choice 1 (\RX) is 
the first to be used up by \RRR.
We  develop \eqref{njord4} as follows.
\begin{lemma}\label{L1}
  We have
  \begin{align}\label{l1}
    S_n-S(T_0)=\max\bigcpar{X_1+2X_2, X_2+2X_3, X_3+2X_1}+\Olpnnn.
  \end{align}
\end{lemma}
\begin{proof}
We may again, by symmetry, suppose that \RRR{} first uses up 1.
Typically, this is the case when $\Xm=X_1$, but it is possible that $X_1$ is
not the maximum. (Then $N_{t,1}$ is not the largest at $t=T_0$, but
$N_{t,1}$ overtakes the other two components and hits $n$ first.)
In any case, 
$N_{T_2,1}=N_{T_1,1}=n$,
and thus \eqref{frej5} yields, recalling $p_1=0$,
\begin{align}\label{njord6}
  0=n-N_{T_2,1}=\Xm-X_1+\Olpnnn.
\end{align}
Hence, 
\begin{align}\label{ull1}
\Xm=X_1+\Olpnnn.
\end{align}
We obtain from \eqref{njord4}, \eqref{frej8} and \eqref{ull1},
recalling $p_2=\frac{1}3$ and $p_3=\frac{2}3$,
\begin{align}\label{njord7}
S_n- S(T_0)&
=3\Xm-\tfrac{2}{3}(T_2-T_1)+\Olpnnn
\notag\\&
=3\Xm-\min\bigcpar{2(\Xm-X_2),(\Xm-X_3)}+\Olpnnn
\notag\\&
=\max\bigcpar{\Xm+2X_2,2\Xm+X_3}+\Olpnnn
\notag\\&
=\max\bigcpar{X_1+2X_2,2X_1+X_3}+\Olpnnn.
\end{align}
Furthermore, \eqref{oden0} implies that $\Xm\ge0$
and that, using also \eqref{ull1},
\begin{align}\label{ull2}
2X_1+X_3& 
= 3X_1+X_2+2X_3 
= 3\Xm+X_2+2X_3 +\Olpnnn
\notag\\&
\ge X_2+2X_3+\Olpnnn.
\end{align}
Hence \eqref{njord7} yields \eqref{l1} in the case when \RRR{} first uses up
choice 1. By symmetry \eqref{l1} holds in general.
\end{proof}

We may now summarize the analysis in the following limit result.

\begin{theorem}\label{TL2}
As \ntoo, we have convergence in distribution, together with all moments,
\begin{align}\label{l2}
    n\qqw S_n \dto 
\SSS:=W +\max\bigcpar{V_1+2V_2, V_2+2V_3, V_3+2V_1},
\end{align}
where
$W,V_1,V_2,V_3$ are jointly normal with $W$ independent of $(V_1,V_2,V_3)$
and 
\begin{align}\label{l2w}
  W &\in N(0,2),
\\\label{l2z}
(V_1,V_2,V_3) & \in N \lrpar{0,
\matrixx{\phantom-\frac23& -\frac13 & -\frac13\\[3pt]
-\frac13& \phantom-\frac23 & -\frac13\\[3pt]
-\frac13& -\frac13 & \phantom-\frac23}
}.
\end{align}
\end{theorem}

\begin{proof}
  The random vectors $\bxi_t$ in \eqref{qN} are \iid{} with $\E \bxi_t=0$
and covariance matrix (regarding $\bxi_t$ as a column vector)
\begin{align}\label{hoder1}
\Var(\bxi_t):= \E \bxi_t\tr\bxi_t - (\E \bxi_t\tr)(\E\bxi_t )
= 
\gS:=
\matrixx{\phantom-\frac29& -\frac19 & -\frac19\\[3pt]
-\frac19&\phantom- \frac29 & -\frac19\\[3pt]
-\frac19& -\frac19 & \phantom-\frac29}.
\end{align}
Since $T_0\sim 3n$ by \eqref{t0},
the central limit theorem yields, recalling \eqref{oden4},
\begin{align}\label{hoder2}
n\qqw  (X_1,X_2,X_3)\dto
  (V_1,V_2,V_3) & \in 
N\xpar{0,3\gS}
,\end{align}
which agrees with \eqref{l2z}.
Similarly, as noted in \eqref{njord0}, $n\qqw S(T_0)\dto W$.
Furthermore, by \eqref{skade1} we may here replace $S(T_0)$ be the
approximation $S'(T_0)$ which, as noted above, 
is independent of $(X_1,X_2,X_3)$.
Hence,
\begin{align}\label{hoder3}
n\qqw  \bigpar{S(T_0),X_1,X_2,X_3}\dto (W,V_1,V_2,V_3),
\end{align}
and thus \eqref{l1} and the continuous mapping theorem yield
\eqref{l2}.
Moreover, all moments converge in the central limit theorems
\eqref{hoder2} and \eqref{njord0} \cite[Theorem 7.5.1]{Gut},
and it  follows (\eg{} using uniform integrability) that all moments converge 
also in \eqref{hoder3} and \eqref{l2}.
\end{proof}
In the following section, we give more convenient expressions for 
the limit $\SSS$.

\section{The distribution of the limit for greedy strategies}
\label{Spf2}

We give several alternative descriptions of the asymptotic distribution
found in \refT{TL2}; using them we then prove \refT{T1}.
See also \refS{Swin} for another use of these descriptions.

\begin{theorem}\label{T2}
The limit $\SSS$ in \refT{TL2}
can be described by any of the following equivalent
formulas:
\begin{romenumerate}
  
\item\label{T2a} 
We have
  \begin{align}\label{t2a}
    \SSS = W + \max\bigcpar{Z_1,Z_2,Z_3}
  \end{align}
where 
$W,Z_1,Z_2,Z_3$ are jointly normal with $W$ independent of\/ $(Z_1,Z_2,Z_3)$
and
\begin{align}\label{t2w}
  W &\in N(0,2),
\\\label{t2z}
(Z_1,Z_2,Z_3) & \in N \lrpar{0,
\matrixx{\phantom-2& -1 & -1\\
-1&\phantom- 2 & -1\\
-1& -1 & \phantom-2}
}.
\end{align}

\item\label{T2b} 
We have
  \begin{align}\label{t2b}
    \SSS = W' + \sqrt3 \max\bigcpar{Z'_1,Z'_2,Z'_3}
  \end{align}
where 
$W',Z'_1,Z'_2,Z'_3$ are independent standard normal  $N(0,1)$.

\item\label{T2c} 
We have
  \begin{align}\label{t2c}
    \SSS = \max\bigcpar{Z''_1,Z''_2,Z''_3}
  \end{align}
where 
$Z_1'',Z''_2,Z''_3$ are jointly normal with 
\begin{align}
\label{t2z''}
(Z''_1,Z''_2,Z''_3) & \in N \lrpar{0,
\matrixx{4&1&1\\
1&4&1\\
1&1&4}
}.
\end{align}

\item \label{T2d}
We have
  \begin{align}\label{t2d}
\SSS =  W+R\cos\gTH,  
  \end{align}
where 
$W,R,\gTH$ are independent with
$W\in N(0,2)$ as in \eqref{t2w}, 
$R$ has a Rayleigh distribution with density
$\frac12 r e^{-r^2/4}$, $r>0$,
and $\gTH$ has a uniform distribution $U(0,\pi/3)$.

\end{romenumerate}
\end{theorem}

We will use the notation
\begin{align}\label{Zm}
  \Zm:=\max\cpar{Z_1,Z_2,Z_3}.
\end{align}
Note also that 
\eqref{t2z} implies that
$Z_1+Z_2+Z_3$ has variance 0, and thus
the normal variables $Z_1,Z_2,Z_3$ in \eqref{t2a} satisfy
$Z_1+Z_2+Z_3=0$ almost surely; thus
$(Z_1,Z_2,Z_3)$ lives in a 2-dimensional space.

\begin{proof}[Proof of \refT{T2}]

\pfitemref{T2a}  
 Define 
 \begin{align}\label{ZV}
Z_1:=V_1+2V_2,\qquad Z_2:=V_2+2V_3, \qquad Z_3=V_3+2V_1.   
 \end{align}
Then \eqref{l2} shows that \eqref{t2a} holds, and 
a simple calculation shows that $(Z_1,Z_2,Z_3)$ has the distribution
\eqref{t2z}.

\pfitemref{T2c}  
Define $Z''_i:=W+Z_i$, $i=1,2,3$.
Then \eqref{t2a} yields \eqref{t2c}, 
and \eqref{t2w}--\eqref{t2z} yield \eqref{t2z''}.

\pfitemref{T2b}
We may write $W=W'+\tW$, where $W',\tW \in N(0,1)$, and $W'$ and $\tW$ are
independent of each other and of $(Z_1,Z_2,Z_3)$.  
Define $Z'_i:=(\tW+Z_i)/\sqrt3$, $i=1,2,3$.
Then \eqref{t2a} yields \eqref{t2b}, 
and it follows from \eqref{t2z} that
the covariance matrix of $(Z'_1,Z'_2,Z'_3)$
is the identity matrix; thus the jointly normal variables
$W',Z_1',Z_2',Z_3'$ are independent $N(0,1)$.

\pfitemref{T2d}
As said above, 
$Z_1+Z_2+Z_3=0$ almost surely, so $(Z_1,Z_2,Z_3)$ 
has really a 2-dimensional normal distribution.
In fact, if $\bzeta=(\zeta_1,\zeta_2)$ is a centered normal distribution in
$\bbR^2$ with $\Var\zeta_1=\Var\zeta_2=2$ and $\Cov(\zeta_1,\zeta_2)=0$,
then we can construct $(Z_1,Z_2,Z_3)$ with the desired
distribution \eqref{t2z} by
\begin{align}\label{freja1}
  Z_i :=  \bff_i\cdot \bzeta,
\end{align}
where $\bff_1:=(1,0)$, $\bff_2:=(-\frac12,\frac{\sqrt3}2)$,
$\bff_3:=(-\frac12,-\frac{\sqrt3}2)$.
We define $R:=|\bzeta|$ and $\gTH:=\arg(\zeta_1+\ii\zeta_2)\in[-\pi,\pi)$;
thus
\begin{align}\label{freja2}
  \bzeta = \xpar{R\cos\gTH,R\sin\gTH}, 
\end{align}
and it follows from \eqref{freja1} by simple calculations
(which are made even simpler by identifying $\bbR^2$ and $\bbC$ and regarding
$\bzeta$ as a complex random variable) that
\begin{align}\label{freja3}
Z_1=R\cos\gTH,
\qquad
Z_2= R\cos(\gTH-2\pi/3),  
\qquad
Z_3= R\cos(\gTH+2\pi/3).
\end{align}
The normal distribution of $\bzeta$ is rotationally symmetric, and thus, as
is well-known, $R$ and $\gTH$ are independent, with $\gTH$ uniformly
distributed
on $[-\pi,\pi)$; furthermore, $R$ has the Rayleigh distribution stated in
the theorem.
To find the distribution of $\Zm:=\max\cpar{Z_1,Z_2,Z_3}$, we may by symmetry 
condition on $\Zm=Z_1$, which by \eqref{freja3} is equivalent to
$\gTH\in[-\pi/3,\pi/3]$, and since $\cos\gTH$ is an even function, we may 
further restrict to $\gTH\in[0,\pi/3]$. 
Then $\Zm=Z_1=R\cos\gTH$, and thus \eqref{t2d} follows from \eqref{t2a}
\end{proof}

\begin{proof}[Proof of \refT{T1}]
By the moment convergence in \refT{TL2}, it suffices to find $\E\SSS$.
For this we use \eqref{t2d}.
We have $\E W=0$, and 
by simple calculations
\begin{align}\label{tyr1}
  \E R& = \intoo \tfrac12 r^2 e^{-r^2/4}\dd r = \sqrt\pi,
\\\label{tyr2}
\E\cos\gTH &=  \frac{3}{\pi}\int_0^{\pi/3}\cos \vartheta\dd\vartheta
= \frac{3\sqrt3}{2\pi}.
\end{align}
Hence, by the independence,
\begin{align}\label{tyr3}
\E\SSS=\E R \cdot \E\cos\gTH
= \frac{3\sqrt3}{2\sqrt{\pi}}
\doteq  1.4658075
.\end{align}
\end{proof}

\begin{remark}
  Alternatively, we can use
\eqref{t2b} and conclude
\begin{align}
\E\SSS = \sqrt3\E\max\bigcpar{Z'_1,Z'_2,Z'_3},
\end{align}
where the \rhs{} contains the expectation of the maximum of three \iid{}
standard normal 
variables which is known to be $3/(2\sqrt\pi)$ \cite{Jones1948}.
\end{remark}

Higher moments of $\SSS$ can be computed in the same way.
For example, we have
\begin{align}\label{ymer1}
  \E (\SSS)^2
=\E W^2 + \E R^2\E\cos^2\gTH
= 2 + 4 \Bigpar{\frac12+ \frac{3\sqrt3}{8\pi}} 
=
4 + \frac{3\sqrt3}{2\pi}
\doteq4.82699
\end{align}
and hence
\begin{align}\label{ymer2}
  \Var \SSS
=4-\frac{27-6\sqrt3}{4\pi}
\doteq2.67840
.\end{align}
Hence, we have
\begin{align}\label{ymer3}
  \Var \Sgr_n \sim \Bigpar{4-\frac{27-6\sqrt3}{4\pi}}n.
\end{align}

\section{Analysis when \NNN{} does not play greedily}\label{SNNN+}
Assume as above that \RRR{} uses the optimal strategy, i.e., the greedy
strategy.
In \refS{Spf} we assumed that \NNN{} uses the greedy strategy.
More generally, suppose now that \NNN{} uses any strategy
that does not do anything stupid (a move that cannot win when \RRR{} has
only one or two choices).
(This includes both the unknown optimal strategy for \NNN, and the greedy
strategy, but also many others.)
Then, as noted in \refSS{SSNNN}, 
the expected gain for \NNN{} is still $1/3$ in each round where \RRR{}
has two choices left, and 1 in each round where \RRR{} has only one choice.
Hence, \eqref{ESn} still holds.
Moreover, the strategy of \RRR{} is not affected by the moves made by \NNN,
and thus the random walk $\NN_0,\dots,\NN_{3n}$ and the variables $T_1, T_2,
X_1,X_2,X_3,\Xm$ (and others) are the same as in \refS{Spf}.
In particular, \eqref{hoder2} still holds.

For the score of \NNN, recall first that in Phase~\ref{A1}, when \RRR{}
still has three choices,  \RRR{} plays each with the same probability.
It follows that regardless of the strategy of \NNN, the outcome $\gD S(t)$
of each round has the same distribution as discussed in \refS{Spf}, i.e.,
$1$, $0$, 
or $-1$ with probability $1/3$ each; moreover, this is independent of the
previous history, so the outcomes of different rounds in this phase are
independent. 
Consequently,  
$S(T_1)$ has the same distribution as for the greedy strategy, and so 
has $S(T_0)$ if we condition on $T_1>T_0$. 
It follows that
\eqref{tor10} still holds, and so do \eqref{njord0}--\eqref{skade1}.
However, there is one important difference from the case of the greedy
strategy in \refS{Spf}: there the score $S(T_0)$ is independent of
$(X_1,X_2,X_3)$ (again conditioned on $T_1>T_0$). This is no longer true in
general, since the strategy of \NNN{} may cause dependencies.
We give a simple example showing that this actually may happen in
\refE{Erock}.

In Phase~\ref{A2}, \RRR{} has two choices, and uses the greedy strategy
described in \refSS{SSgreedy}\ref{RRR2}.
We have assumed that the strategy of \NNN{} is not stupid, and that leaves
two choices for \NNN. Both give an expected gain $\E\gD S(t)=1/3$, but the
distributions are different.
The precise distribution of $S(T_2)-S(T_1)$ may therefore depend on the
strategy of \NNN, 
but if we define $M_i:=S(T_1+i)-S(T_1)-\frac13 i$, then the sequence
$(M_{i\bmin(T_2-T_1)})_{i\ge1}$ (where we stop at $T_1+i=T_2$)
is,
for any non-stupid strategy of \NNN, 
a martingale with uniformly bounded increments, and Doob's inequality
shows that \eqref{njord2} holds.

In Phase~\ref{A3}, \NNN{} has only one choice that is not stupid, so the
strategy is the same is in \refS{Spf}, and \eqref{njord3} still holds.

It follows that \eqref{njord4} holds, and thus
\refL{L1} holds, by the same proof as above.
This leads to the following result.

\begin{theorem}\label{TN}
Suppose that \RRR{} uses the optimal (i.e., greedy) strategy, and that
\NNN{} uses any non-stupid strategy. (For example, his optimal strategy.)
If we decompose
\begin{align}\label{virgo}
n\qqw S_n=n\qqw S(T_0)+ n\qqw(S_n-S(T_0)),
\end{align}
then the two terms individually converge in distribution to the limits
$W$ and $\Zm$ in \eqref{t2a};
however, in general the two terms are  dependent, 
so \refTs{TL2} and \refT{T2} do not hold.
\end{theorem}

Note that it does not follow from \refT{TN} that $n\qqw S_n$ converges in
distribution. By general principles, the convergence in distribution implies
that each of the sequences $n\qqw S(T_0)$  and $n\qqw(S_n-S(T_0)$ is tight,
and thus so is their sum $n\qqw S_n$; this implies that there are
subsequences that converge in distribution, but it is conceivable that
different subsequences have different limits. 
(This can easily happen if the strategy explicitly depends on, for example,
whether $n$ is even or odd, but it is not expected for ``natural''
strategies.)

\begin{remark}\label{Rmaria}
  In general, any (subsequential) limit in distribution $\cS$ can be written as 
$W+\Zm$ with $W$ and $\Zm$ as in \refT{T2}, but possibly dependent.
It follows from Minkowski's inequality and calculations as in
\eqref{ymer1}--\eqref{ymer2} that, 
with the notation $[a\pm b]:=[a-b,a+b]$,
\begin{align}\label{beata}
(\Var\cS)\qq \in \bigsqpar{(\Var W)\qq \pm (\Var \Zm)\qq}
=\Bigsqpar{\sqrt2\pm \sqrt{2-\frac{27-6\sqrt3}{4\pi}}}
\end{align}
and thus, numerically,
$  (\Var\cS)\qq \in [0.590\dots,2.237\dots]$ and thus
\begin{align}\label{maria}
  \Var\cS \in [0.348\dots,5.008\dots].
\end{align}
Since we have moment convergence by the same arguments as before, it follows
that, for any non-stupid strategy for \NNN,
$\liminf n\qw\Var S_n$ and $\limsup n\qw\Var S_n$ lie in the interval
\eqref{maria}. Furthermore, \eqref{maria} shows that $\Var\cS>0$, so the
limit distribution is non-degenerate.
\end{remark}

We give next a simple example showing that there are strategies for \NNN{}
for which
$n\qqw S_n$ has a limit in distribution that is different from $\SSS$;
we then discuss briefly  the optimal strategy.

\begin{example}\label{Erock}
  Let the strategy of \NNN{} be to always play \RX{} as long as \RRR{} has
  three choices, and then switch to the greedy strategy for the endgame.
(This is obviously a risky strategy if \RRR{} would guess it, but we assume
that \RRR{} is a mathematician and knows that the greedy strategy is proven
to be optimal, and therefore sticks to it.)
We do not claim that this is a clever strategy, but it is not stupid in the
sense above; thus the results above hold for it.
Moreover, in Phase~\ref{A1}, \NNN{} wins when $\RRR$ plays \SX, and loses
when \RRR{} plays \PX; hence $S(t)=N_{t,3}-N_{t,2}$ for all $t\le T_1$.
Consequently, assuming $T_1>T_0$, we have
\begin{align}
  S(T_0)=N_{T_0,3}-N_{T_0,2}=X_3-X_2.
\end{align}
It follows that \eqref{l2} 
still holds, 
with $(V_1,V_2,V_3)$ and $(Z_1,Z_2,Z_3)$ as before and
\begin{align}\label{Wrock}
  W=V_3-V_2=-V_1-2V_2=-Z_1.
\end{align}
Hence, instead of \eqref{l2} and \eqref{t2a} we find
\begin{align}\label{Zrock}
  n\qqw S_n \dto \cS := W+\Zm=-Z_1+\Zm
=\max\cpar{0, Z_2-Z_1,Z_3-Z_1}.
\end{align}
Note that \eqref{l2w}--\eqref{l2z} and \eqref{t2w}--\eqref{t2z} still hold,
but $W$ and $\Zm$ are no longer independent. To see that the dependence
really matters and leads to a different limit distribution $\cS$ than for
the greedy strategy,
we compute, using symmetry and the representation in \refT{T2}\ref{T2d},  
\begin{align}
  \E\bigsqpar{W^2\Zm}&
=\E\bigsqpar{Z_1^2\Zm}
=\tfrac13\E\Bigsqpar{\Zm\sum_1^3Z_i^2}
\notag\\&
=\tfrac13\E\bigsqpar{(R\cos\gTH) R^2}
=\tfrac13\E R^3\, \E\cos\gTH
\notag\\&
>
\tfrac13\E R^2\cdot\E R \E\cos\gTH
=\E W^2 \E \Zm.
\end{align}
($\E R^3>\E R^2\E R$ follows from Lyapounov's inequality, or because a
calculation yields
$\E R=\sqrt\pi$, $\E R^2=4$, $\E R^3=6\sqrt\pi$.)
Similarly,
\begin{align}
  \E\bigsqpar{W\Zm^2}
=-\E\bigsqpar{Z_1\Zm^2}
=-\tfrac13\E\bigsqpar{\Zm^2\sum_1^3Z_i}
=0
=\E W \E \Zm^2.
\end{align}
It follows that if $W'\sim N(0,2)$ is independent of $\Zm$, then
\begin{align}
  \E \cS^3 = \E (W+\Zm)^3 > \E(W'+\Zm)^3 = \E (\SSS)^3.
\end{align}
Hence the limit distribution $\cS$ differs from $\SSS$ for the greedy
distribution.
\end{example}

\subsection{On the optimal strategy for \NNN}\label{SSNNNop}

Consider now the unknown optimal strategy for $\NNN$. 
\refT{T2} leads to  an obvious conjecture:
\begin{conj}
  If both players play optimally, then
  \begin{align}
    n\qqw \Sop_n \dto \SSSop=W+\Zm,
  \end{align}
where $W$ and $\Zm:=\max\cpar{Z_1,Z_2,Z_3}$ each are as in \refT{T2}, but
they now may be dependent.
\end{conj}
Note that if this holds, then \eqref{beata}--\eqref{maria} hold for $\SSSop$.

The optimal strategy for \NNN{} has to punish
strategies for \RRR{} like the one in \refE{ESpiro}.
Intuitively, it therefore seems likely that if \RRR{} plays greedily,
then the optimal strategy of \NNN{} will punish \RRR{} in games
where the times $T_1$ and $T_2$ in our analysis in \refS{Spf} are unusually
large (and conversely reward \RRR{} when they are small; remember that the
expectation is the same as if \NNN{} plays greedily).
It therefore seems likely that if both players play optimally, 
there is a negative correlation between the two terms in \eqref{virgo}.
However, even if this is correct,
it is possible that the dependency vanishes asymptotically so that
we have the same limit $\SSS$ as in \refT{T2}. We have no guess, and leave
this as a problem.

\begin{problem}
  If both players play optimally, does $n\qqw S_n$ have the same asymptotic
  distribution $\SSS$ as in \refT{T2} for greedy play?
If not, is there an asymptotic distribution $\SSSop$ (as conjectured above), 
and what is it?
\end{problem}

\section{The probability of winning for greedy play}\label{Swin}
Finally, we return to the case of both players using their greedy strategies
and note that we  may also calculate the asymptotic probability
that \RRR{} wins the game,
in spite of her restriction,
\ie, that the final score $S_n<0$.
(Recall that $S_n$ is the score for \NNN.)

\begin{theorem}\label{Twin}
If both players use their greedy strategies, then
the probability that \RRR{} wins 
has as \ntoo{} the limit
\begin{align}\label{tp}
  \P(S_n<0)
\to
\frac{3\arccos(1/4)-\pi}{4\pi}
=
\frac{\arccos(11/16)}{4\pi}
\doteq 0.064677.
\end{align}
\end{theorem}

\begin{proof}
  By \refT{T2}, we have $\P(S_n<0)\to \P(\SSS<0)$
(since $\SSS$ has a continuous distribution, \eg{} by \eqref{t2a}).
We compute this probability using \refT{T2}\ref{T2c}.
By \eqref{t2c}, we have 
\begin{align}\label{loke1}
\SSS<0 \iff Z''_i<0\ \forall i.
\end{align}
We may, similarly to \eqref{freja1}, construct $Z''_i$ as
\begin{align}\label{loke2}
  Z''_i := \hbf_i\cdot \hbzeta,
\end{align}
where $\hbzeta$ is a standard normal distribution in $\bbR^3$, and
$\hbf_1,\hbf_2,\hbf_3$ are three vectors in $\bbR^3$ such that 
\begin{align}\label{loke3}
  \hbf_i\cdot\hbf_j =
  \begin{cases}
    4,& i=j,
\\
-1,& i\neq j.
  \end{cases}
\end{align}
By \eqref{loke2}, the condition \eqref{loke1} means that $\hbzeta$ lies in
the intersection of three open half-spaces $H_1$, $H_2$, $H_3$, 
which are bounded by hyperplanes
orthogonal to $\hbf_1$, $\hbf_2$ and $\hbf_3$. The angle between any two of
these vectors is, by \eqref{loke3}, $\ga:=\arccos(-1/4)$. 
Hence, the interior angle between any of the two hyperplanes is 
$\beta:=\pi-\ga =\arccos(1/4)$, and thus the intersection of 
the unit sphere and $H_1\cap H_2\cap H_3$ is a spherical triangle $\gD$ with all
three angles $\beta$. 
Consequently, the area $|\gD|$ of $\gD$ is $3\gb-\pi$.
The distribution of $\hbzeta$ is rotationally symmetric, and thus we may
project $\hbzeta$ onto the unit sphere, and find,
recalling that the area of the sphere is $4\pi$,
\begin{align}
  \P(\Zm<0)& = \P\bigpar{\hbzeta\in H_1\cap H_2\cap H_3} 
= \frac{|\gD|}{4\pi}
= \frac{3\gb-\pi}{4\pi}
= \frac{3\arccos(1/4)-\pi}{4\pi}
.\end{align}
Finally, note that 
\begin{align}
\cos(3\gb-\pi)=-4\cos^3\gb+3\cos\gb=-4\bigpar{\tfrac14}^3+3\cdot\tfrac{1}{4}=
\tfrac{11}{16}.  
\end{align}
\end{proof}

\refT{Twin} assumes that the players use their greedy strategies;
we know that this is optimal for \RRR{}, and yields the same expectation for
\NNN{} as his optimal strategy, if their objectives are to maximize the
expected gain; if they instead want to maximize the 
probability of winning (but do not care about how much they win or lose), 
the optimal strategies are presumably different (see \refE{E1}), and
most likely much more complex; hence we do not know whether \eqref{tp} holds or
not in that case.

\begin{problem}\label{P1}
  Suppose that both players want to maxime $\P(\text{win})-\P({\text{lose}})$.
What is (asymptotically) the probability that \RRR{} wins?
\end{problem}
It is possible that the asymptotic answer is the same as in
\refT{Twin}, although the 
probabilities for finite $n$ are different.
(See \refE{E1}.)
It might seem likely that a strategy that gives one of the players
a significantly lower expected score
will also give  a lower probability that this score is positive.
However, \refE{E2} shows that strategies with the same expectation still
might give different distributions of the score and therefore different
probabilities of winning, so it seems that there is no simple solution to
\refP{P1}. 

\begin{example}\label{E1}
Here is simple example showing that
the greedy strategy is not the optimal strategy for \RRR{}
if the objective is to win, as in \refP{P1}.
Let $n=2$, and suppose that in the first four rounds,
\RRR{} has (by chance) chosen \RX, \PX, \SX, \SX, 
and that \RRR{} won two of these while two were draws.
Thus the score (for \NNN) $S(4)=-2$.
Hence, \NNN{} cannot win, but since he will win the last round,
the game will be a draw if he wins round 5.
Therefore, in round 5, the objective for \RRR{} is to minimize the
probability of losing (but a draw is as good as a win).
In this round \RRR{} plays the game in \refF{fig:1};
if she wants to minimize the probability of losing this round
the best strategy is to play \RX{} or \PX{} with equal probabilities,
and not with  the probabilities in \refSS{SSgreedy} that minimize the
expected loss.
(The example can be extended to any $n\ge2$ by assuming that \RRR{} has
played the three choices $n-2$ times each in the first $3(n-2)$ rounds, and
that each of these rounds was a draw; the play then continues as above.)
\end{example}

\begin{example}\label{E2}
  Suppose that \RRR{} uses the greedy strategy above, but that \NNN{}
uses the strategy in \refE{Erock}.
As seen in \refE{Erock},
then
\begin{align}\label{hel}
  n\qqw S_n \dto \cS:=-Z_1+\Zm  
=\max\bigcpar{0, Z_2-Z_1,Z_3-Z_1}.
\end{align}
Thus $\cS\ge0$ with a point mass
 $\P(\cS=0)=1/3$ (by symmetry).
In this case, we cannot immediately find the limit of $\P(S_n<0)$, but
if the strategy is perturbed a little, and \NNN{} plays normally for the
first $\eps_n n$ rounds with $\eps_n\to0$ very slowly, it can be seen that
$\P(S_n<0)\to\frac12\P(\cS=0)=1/6$.

In this case, the new strategy for \NNN{} is worse for him;
it gives the same expected score but a lower probability that the score is
positive (given that \RRR{} plays greedily).
However, it suggests that there also might be other strategies that  instead
increase the probability that \NNN{} wins.
\end{example}

\newcommand\AAP{\emph{Adv. Appl. Probab.} }
\newcommand\JAP{\emph{J. Appl. Probab.} }
\newcommand\JAMS{\emph{J. \AMS} }
\newcommand\MAMS{\emph{Memoirs \AMS} }
\newcommand\PAMS{\emph{Proc. \AMS} }
\newcommand\TAMS{\emph{Trans. \AMS} }
\newcommand\AnnMS{\emph{Ann. Math. Statist.} }
\newcommand\AnnPr{\emph{Ann. Probab.} }
\newcommand\CPC{\emph{Combin. Probab. Comput.} }
\newcommand\JMAA{\emph{J. Math. Anal. Appl.} }
\newcommand\RSA{\emph{Random Structures Algorithms} }
\newcommand\DMTCS{\jour{Discr. Math. Theor. Comput. Sci.} }

\newcommand\AMS{Amer. Math. Soc.}
\newcommand\Springer{Springer-Verlag}
\newcommand\Wiley{Wiley}

\newcommand\vol{\textbf}
\newcommand\jour{\emph}
\newcommand\book{\emph}
\newcommand\inbook{\emph}
\def\no#1#2,{\unskip#2, no. #1,} 
\newcommand\toappear{\unskip, to appear}

\newcommand\arxiv[1]{\texttt{arXiv}:#1}
\newcommand\arXiv{\arxiv}

\newcommand\xand{and }
\renewcommand\xand{\& }

\def\nobibitem#1\par{}

\end{document}